\documentclass[leqno,11pt]{amsart}

\usepackage{bm,cancel,ulem}
\usepackage{amssymb,amsfonts}
\usepackage{amsmath,latexsym}
\usepackage{pdfsync}
\usepackage{xcolor}
\usepackage{mathrsfs}

\usepackage[unicode,breaklinks=true,colorlinks=true]{hyperref}
\newtheorem{thm}{Theorem}[section]

\newtheorem{lem}[thm]{Lemma}
\newtheorem{prop}[thm]{Proposition}
\theoremstyle{definition}
\newtheorem{defn}[thm]{Definition}

\numberwithin{equation}{section}

\newcommand{\R}{\mathbb R}
\newcommand\E{\mathcal{E}}

\newcommand{\dy}{\,{\rm d}y}
\newcommand{\dx}{\,{\rm d}x}
\newcommand{\dt}{\,{\rm d}t}
\newcommand{\ds}{\,{\rm d}s}

\def\Xint#1{\mathchoice	
	{\XXint\displaystyle\textstyle{#1}}	{\XXint\textstyle\scriptstyle{#1}}	{\XXint\scriptstyle\scriptscriptstyle{#1}}	{\XXint\scriptscriptstyle\scriptscriptstyle{#1}}	\!\int}
\def\XXint#1#2#3{{\setbox0=\hbox{$#1{#2#3}{\int}$}	\vcenter{\hbox{$#2#3$}}\kern-.5\wd0}}

\def\dashint{\Xint-}

\newcommand{\crb}[1]{{#1}}

\newcommand{\dis}{\displaystyle}
\newcommand{\lec}{\lesssim}
\newcommand{\norm}[1]{\left\|#1\right\|}
\begin{document}
\thispagestyle{empty}

\vspace{1 true cm} {
\title[local regularity of axisymmetric solution]
 {Remarks on local regularity of axisymmetric solutions to the 3D Navier--Stokes equations}%
\author[H. Chen]{Hui Chen}%
\address[H. Chen]
 {School of Science, Zhejiang University of Science and Technology, Hangzhou, 310023, People's Republic of China }
\email{chenhui@zust.edu.cn}
\author[T.-P. Tsai]{Tai-Peng Tsai}
\address[T.-P. Tsai]
{Department of Mathematics, University of British Columbia, Vancouver, BC V6T1Z2, Canada }
\email{ttsai@math.ubc.ca}
\author[T. Zhang]{Ting Zhang*}

\address[T. Zhang]{School of Mathematical Sciences, Zhejiang University,  Hangzhou 310027, People's Republic of China}

\email{zhangting79@zju.edu.cn}
\thanks{$^*$Corresponding author.}

\begin{abstract}
In this note, a new local regularity criteria for the axisymmetric solutions to the 3D Navier--Stokes equations is investigated. It is slightly supercritical and implies an upper bound for the oscillation of $\Gamma=r u^{\theta}$: for any $0< \tau<1$, there exists a constant $c>0$,
$$
|\Gamma(r,x_{3},t)|\leq N e^{-c\, |\ln r|^{\tau}},\ 0<r\leq \frac{1}{4}.
$$
\end{abstract}

\maketitle



\vskip 0.2cm


\section{Introduction}
\label{intro}
In this paper, we discuss potential singularities of axisymmetric solutions to the incompressible 3D Navier--Stokes equations. Roughly speaking, we would like to show that if slightly supercritical quantities of an axisymmetric solution are bounded, then such a solution is smooth. In order to present it precisely, let us first recall the basic notions from the mathematical theory of the Navier--Stokes equations.

In Cartesian coordinates, the incompressible 3D Navier--Stokes equations are given by:
\begin{equation}\label{NS}
\partial_{t}\bm{u}+\left(\bm{u}\cdot\nabla\right) \bm{u}-\Delta  \bm{u}+\nabla \Pi=0,\quad\nabla\cdot  \bm{u}=0,
\end{equation}
here $\bm{u}(x,t)$ and $\Pi(x,t)$ denote the fluid velocity field and the pressure, respectively. A global weak solution with finite energy was constructed by Leray \cite{Leray1934} and Hopf \cite{Hopf1951}.  However, the uniqueness and regularity of such weak solution is still one of the most challenging open problems in the field of mathematical fluid mechanics.
One essential work is usually referred as Ladyzhenskaya--Prodi--Serrin conditions (see \cite{Escauriaza2003,Prodi1959,Seregin2011,Serrin1962,Takahashi1990} and the references therein), i.e. if the weak solution $\bm{u}$  satisfies
\begin{equation} \bm u\in L^{p}(0,T; L^{q}(\R^{3})),\ \
	\frac{2}{p}+\frac{3}{q}=1,\ 3\leq q\leq \infty,
\end{equation} then the weak solution is regular in $(0,T]$. \crb{Its endpoint $q=3$ is proved in the celebrated paper \cite{Escauriaza2003} of Escauriaza, Seregin, and \v{S}ver\'{a}k.}
In a remarkable recent development \cite{Tao2019}, Tao  used a new approach to provide the explicit quantitative estimates for solutions of the Navier--Stokes equations belonging to the critical space $L^{\infty}\left(0, T ; L^{3}\left(\R^{3}\right)\right)$. As a consequence of these quantitative estimates, Tao showed that if the solution $\bm{u}$ first blows up at $T_{*}>0$, then for some absolute constant $\alpha>0$,
\begin{align}
\limsup _{t \rightarrow T_{*}} \frac{\|\bm{u}(\cdot, t)\|_{L^{3}\left(\R^{3}\right)}}{\left(\ln \ln \ln \frac{1}{T_{*}-t}\right)^{\alpha}}=\infty.
\end{align}
T. Barker and C. Prange \cite{Barker2021} and S. Palasek \cite{Palasek2021} made progress on removing some logarithms from the blow-up rate.

A particular class of weak solutions to \eqref{NS} called \emph{suitable weak solutions} is introduced by L. Caffarelli, R. Kohn  and L. Nirenberg in their celebrated paper \cite{Caffarelli1982}. A simple proof is also given by F. Lin in \cite{Lin1998}; \crb{see Ladyzhenskaya and Seregin \cite{LS99} for revision.}
\begin{defn}
We say that  $(\bm{u}, \Pi)$ is a suitable weak solution of \eqref{NS} in an open domain $\Omega_{T}=\Omega \times(-T, 0)$, $T>0$, if
\begin{enumerate}
	\item[$(1)$] $\bm{u} \in L^{\infty}\left(-T, 0 ; L^{2}(\Omega)\right) \cap L^{2}\left(-T, 0 ; H^{1}(\Omega)\right)$ and $\Pi \in L^{\frac{3}{2}}\left(\Omega_{T}\right)$;
	\item[$(2)$] \eqref{NS} is satisfied in the sense of distributions;
	\item[$(3)$] the local energy inequality holds: for any nonnegative test function $\varphi \in C_{c}^{\infty}\left(\Omega_{T}\right)$  and  $ t \in (-T, 0)$,
	\begin{align}\label{local energy}
		&\int_{\Omega}|\bm{u}(x, t)|^{2} \varphi ~\text{d}x+2 \int_{-T}^{t} \int_{\Omega}|\nabla \bm{u}|^{2} \varphi ~\text{d}x\text{d}s \notag\\
		\leq& \int_{-T}^{t} \int_{\Omega}|\bm{u}|^{2}\left(\partial_{s} \varphi+\Delta \varphi\right)+\bm{u} \cdot \nabla \varphi\left(|\bm{u}|^{2}+2 \Pi\right) ~\textrm{d}x\textrm{d}s.
	\end{align}
\end{enumerate}
\end{defn}
We introduce the conventional notations in local regularity theory for suitable weak solution. Let $B(x_{0},R)$ be the ball in $\R^3$ with center at $x_{0}$ and radius $R$;  $Q(z_{0},R)=B(x_{0},R)\times(t_{0}-R^2,t_{0})$ with $z_{0}=(x_{0},t_{0})$; and $L^{p,q}(Q(z_{0},R))=L^{q}\left(t_{0}-R^2,t_{0};L^{p}(B(x_{0},R))\right)$. For a given solution $(\bm{u}, \Pi)$, let
\begin{align}
	&A(z_{0},R)=\sup_{t_{0}-R^{2}<t<t_{0}} \frac1R \int_{B(x_{0},R)}|\bm{u}(x, t)|^{2} \dx, \quad 
	 E(z_{0},R)=\frac1R \iint_{Q(z_{0},R)}|\nabla \bm{u}|^{2} \dx\dt, \nonumber
	 \\  \label{AE.def}
	&C(z_{0},R)=R^{-2}\iint_{Q(z_{0},R)}|\bm{u}|^3\dx\dt, \quad D(z_{0},R)=R^{-2}\iint_{Q(z_{0},R)}|\Pi|^{\frac{3}{2}}\dx\dt,\\
	&\E(z_{0},R)=A(z_{0},R)+E(z_{0},R)+D(z_{0},R). \nonumber
\end{align}
S. Gustafson, K. Kang and T. Tsai \cite{Gustafson2007} proved  general  $\varepsilon$-regularity criteria for Navier--Stokes equations. They showed that a suitable weak solution $(\bm{u},\Pi)$ is regular at $z_{0}$ if for some small $\varepsilon>0$,
\begin{align}\label{cmp}
\limsup_{R\rightarrow0}\, R^{1-\frac{3}{p}-\frac{2}{q}}\|\bm{u}\|_{L^{p,q}(Q(z_{0},R))}\leq \varepsilon,
\end{align}
where $\frac{3}{p}+\frac{2}{q}\leq 2$.

In our standing assumption, it is supposed that a suitable weak solution $\left(\bm{u},\Pi\right)$ of  Navier--Stokes equations \eqref{NS} is axially symmetric with respect to the axis $x_{3}$. It means that in the corresponding cylindrical coordinate system,
\begin{equation}\label{u}
	\bm{u}(x,t)=u^{r}(r,x_{3},t)\bm{e}_{r}+u^{\theta}(r,x_{3},t)\bm{e}_{\theta}+u^{3}(r,x_{3},t)\bm{e}_{3},\quad \Pi(x,t)=\Pi(r,x_{3},t)
\end{equation}
where $x=\left(x_{1},x_{2},x_{3}\right),r=\sqrt{x_{1}^{2}+x_{2}^{2}}$, and
\begin{equation*}
	\bm{e}_{r}=\left(\frac{x_{1}}{r},\frac{x_{2}}{r},0\right),~\bm{e}_{\theta}=\left(-\frac{x_{2}}{r},\frac{x_{1}}{r},0\right),~\bm{e}_{3}=(0,0,1).
\end{equation*}
The angular velocity $u^{\theta}$ is usually called the swirl.
For the axisymmetric solution $\left(\bm{u},\Pi\right)$, we can equivalently reformulate \eqref{NS} as
\begin{equation}\label{ANS}
	\left\{
	\begin{aligned}
		&\partial_{t}u^{r}+\left(\bm{b}\cdot\nabla\right)u^{r}-\left(\Delta-\frac{1}{r^{2}}\right)u^{r}-\frac{(u^{\theta})^{2}}{r}+\partial_{r}\Pi=0,\\
		&\partial_{t}u^{\theta}+\left(\bm{b}\cdot\nabla\right)u^{\theta}-\left(\Delta-\frac{1}{r^{2}}\right)u^{\theta}+\frac{u^{\theta}u^{r}}{r}=0,\\
		&\partial_{t}u^{3}+\left(\bm{b}\cdot\nabla\right)u^{3}-\Delta u^{3}+\partial_{3}\Pi=0,\\
		&\nabla\cdot\bm{b}=0,
	\end{aligned}
	\right.
\end{equation}
where $\bm{b}=u^{r}\bm{e}_{r}+u^3\bm{e}_{3}$. Define the quantity $\Gamma=ru^{\theta}$, which satisfies
\begin{align}\label{Gamma}
	\partial_{t} \Gamma+\left(\bm{b}\cdot\nabla\right)\Gamma-\left(\Delta-\frac{2}{r}\partial_{r}\right)\Gamma=0.
\end{align}
The global well-posedness was firstly investigated under no swirl assumption (i.e. $u^{\theta}=0$), independently by O. Ladyzhenskaya \cite{Ladyzhenskaya1968} and M. Ukhovskii and V. Yudovich \cite{Ukhovskii1968}, see also Leonardi etc. \cite{Leonardi1999} for a refined proof. When the swirl $u^{\theta}$ is not trivial, it is still open. There are many papers on regularity of axisymmetric solutions, please refer to \cite{Chen2009,Chen2008a,Chen2017,Chen2019,Koch2009,Lei2011a,Lei2017,Palasek2021,Pan2016,Seregin2020,Seregin2021,Wei2016} and the references therein. For the convenience of readers, we list some related results here. 

In regard to the regularity criteria only involving swirl component $u^\theta$, one of the primary results is given by H. Chen, D. Fang and T. Zhang \cite{Chen2017}: the solution $\bm{u}$ is smooth in $(0,T)$, provided that
\begin{equation}
	r^{d} u^{\theta} \in L^{q}((0,T);L^{p}(\R^{3})),
\end{equation}
where $\frac{2}{q}+\frac{3}{p}\leq1-d,~0\leq d<1,~\frac{3}{1-d} < p\leq \infty, ~\frac{2}{1-d} \leq q\leq\infty$. Z. Lei and Q. Zhang \cite{Lei2017} obtained the regularity of the solution under the condition
\begin{equation}\label{Lei}
	r|u^{\theta}| \leq  N |\ln{r}|^{-2},\quad 0<r\leq\frac{1}{2},
\end{equation}
here $N$ is the general constant throughout this paper. D. Wei \cite{Wei2016} improved it to 
\begin{equation}\label{Wei}
	r|u^{\theta}| \leq  N |\ln{r}|^{-\frac{3}{2}},\quad 0<r\leq\frac{1}{2}.
\end{equation}
Chen, Strain, Yau and Tsai \cite{Chen2008a,Chen2009} and Koch, Nadirashvili, Seregin and Sverak \cite{Koch2009} proved that the suitable weak solution is smooth if the solution $\bm{u}$ satisfies 
\begin{align}\label{chen}
|\bm{u}|\leq N r^{-1+\varepsilon}t^{-\frac{\varepsilon}{2}},
\end{align}
where $0\leq\varepsilon\leq1$. Z. Lei and Q. Zhang \cite{Lei2011a} obtained a similar results if
\begin{align}\label{lei}
\bm{b}\in L^\infty((0,T);BMO^{-1}(\R^3)).
\end{align}
A local regularity condition for axisymmetric solutions is proved by G. Seregin \cite[Theorem 2.1]{Seregin2020}, which says that suitable weak solution $(\bm{u},\Pi)$ is regular at $z_{0}=(0,x_{3},t)$ if
\begin{align}\label{seregin1}
\min\left\{\limsup_{R\rightarrow0}A(z_{0},R),\limsup_{R\rightarrow0}E(z_{0},R),\limsup_{R\rightarrow0}C(z_{0},R)\right\}<\infty.
\end{align}
Compared with the $\varepsilon-$regularity \eqref{cmp}, 
the one here in axisymmetric system does not need the  smallness of those scale-invariant energy quantities. It reads that the axisymmetric solution have no Type I singulaities.

Recently, X. Pan \cite[Theorem 1.1]{Pan2016} proved the regularity of the solution under a slightly supercritical assumption
\begin{align}\label{pan}
r|\bm{u}|\leq N \left(\ln\ln \frac{100}{r}\right)^{0.028},\quad 0<r\leq \frac{1}{2}.
\end{align}
G. Seregin \cite{Seregin2021} obtained another type of supercritical regularity criterion:
\begin{align}\label{seregin2}
R^{-\frac{1}{2}}\|\bm{u}\|_{L^{3,4}(Q(z_{0},R))}+R^{-\frac{1}{2}}\|\bm{u}\|_{L^{\frac{10}{3}}(Q(z_{0},R))}\leq N \left(\ln\ln \frac{100}{R}\right)^{\frac{1}{224}},
\end{align}
for any $0<R<1$ and $z_{0}=(0,x_{0,3},t_{0})\in \R^3\times(0,T)$. The key ingredient is that under the assumption \eqref{pan} or \eqref{seregin2}, the bound for  oscillation of  $\Gamma(x,t)$ contained a logarithmic factor.

Using the strategy of quantitative estimates introduced by T. Tao \cite{Tao2019}, S. Palasek \cite[Theorem 2]{Palasek2021} showed that if the axisymmetric  solution $\bm{u}$ first blows-up at $T_{*}>0$ and $2<p\leq3$, then for a constant $\alpha>0$ depending only on $p$,
\begin{align}\label{palasek}
	\limsup_{t \rightarrow T_{*}}\frac{\|r^{1-\frac{3}{p}} \bm{u}\|_{L^{p}(\R^3)}}{\left(\ln \ln \frac{1}{T_{*}-t}\right)^{\alpha}}=\infty .
\end{align}

Let $Q(r) = Q(0,r)$ for $0\in \R^4$ and $r>0$.
Let  the weight 
\[
\omega(R)=\left(\ln\ln \frac{100}{R}\right)^{-1}.
\]

In this paper, we obtain \crb{the following key} proposition.
\begin{prop}\label{lemA6}
	Assume that	$(\bm{u},\Pi)$ is an axisymmetric suitable weak solution to the Navier--Stokes equations in $Q(1)$ and there exist constants $\beta\in (0,\frac18)$ and $K>0$ such that
	\begin{align}\label{A}
		A(z_{0},R) \omega(R)^{\beta}\leq K,
	\end{align} 
	for all $0<R\leq \frac{1}{4}$, for some $z_{0}=(0,x_{0,3},t_{0})\in Q(\frac{1}{8})$. Then for any $0 <\tau<1$, there exists a constant $c=c(K,\beta,\tau)$  such that
	\begin{align}\label{OSC}
		\rm{OSC}_{(x,t)\in Q(z_{0},\rho)}\, \Gamma(x,t)\leq e^{-c\,\left(\left(\ln\frac{100}{\rho}\right)^{\tau}-\left(\ln \frac{100}{R}\right)^\tau-2\right)}\ \rm{OSC}_{(x,t)\in Q(z_{0},R)}\, \Gamma(x,t),
	\end{align}
	for $0<\rho<R\leq \frac{1}{4}$. Here the oscillation $\rm{OSC}_{(x,t)\in Q(z_{0},R)}\, \Gamma(x,t)=\sup_{(x,t)\in Q(z_{0},R)}\,\Gamma(x,t)-\inf_{(x,t)\in Q(z_{0},R)}\,\Gamma(x,t)$. 
	
	If in addition, assuming that
	\eqref{A} holds for \emph{all} $z_{0}=(0,x_{0,3},t_{0})\in Q(\frac{1}{8})$, then we have that for $0<r\leq \frac{1}{4}, |x_{3}|<\frac{1}{8},-\frac{1}{64}< t<0$,
	\begin{align}\label{lnr}
	|\Gamma(r,x_{3},t)|\leq N e^{-c\, |\ln r|^{\tau}}.
	\end{align}
\end{prop}
It should be pointed out that decay in \eqref{lnr}  is an improvement of the result in \cite[Theorem 1.2]{Pan2016}. 

\crb{\textit{Remark.} When this paper was near completion, Professor Seregin posted a similar result \cite{Seregin2021} to arXiv. His assumption is still \eqref{seregin2} and his decay exponent $\tau$ is $1/4$. Our $\tau$ is limited to $\tau<1$, weaker than the H\"older continuity case $\tau=1$.}  

\crb{The following theorems are corollaries of Proposition \ref{lemA6}.}

Fix $1< p,q \leq \infty$ with $\frac{3}{p}+\frac{2}{q}=2-\gamma,\ 0<\gamma<1$. 
Denote 	
\begin{equation}\label{GR.def}
G(z_{0},R)=R^{1-\frac{3}{p}-\frac{2}{q}}\|\bm{b}\|_{L^{p,q}(Q(z_{0},R))} ,\quad  G_{\alpha}(z_{0},R)=G(z_{0},R)\crb{\omega(R)^{\alpha}},
\end{equation}
 where  $0<\alpha<\alpha_{0}= \frac{\gamma}{48+16\gamma}$. 
 
\begin{thm}\label{thm1}
Assume that	$(\bm{u},\Pi)$ is an axisymmetric suitable weak solution to the Navier--Stokes equations in $Q(1)$. If there exist a positive constant $G$  such that
\begin{align}\label{G}
 G_{\alpha}(z_{0},R) \leq G,
\end{align}
for all $z_{0}=(0,x_{0,3},t_{0})\in Q(\frac{1}{8})$ and $0<R\leq \frac{1}{4}$, then the solution is regular at $(0,0)$.
\end{thm}
The constant $\alpha_{0}=\frac{\gamma}{48+16\gamma}$ is not optimal and may be improved if choosing $p,q$ specifically. We leave it to the interested readers. 
\crb{The regularity criterion \eqref{G} would imply criteria} \eqref{pan}, \eqref{seregin2} and \eqref{palasek}, if one \crb{could increase the exponent} $\alpha$. 

More supercritical regularity criteria for the stream function $\psi^{\theta}$ with $\bm{b}=\nabla\times\left(\psi^{\theta}\bm{e}_{\theta}\right)$ or $\omega^{\theta}=\partial_{3}u^{r}-\partial_{r}u^{3}$ or one component $u^{3}$ can be obtained, if we make some minor modifications to Lemma \ref{lem2}, see \cite{Chen2017,Gustafson2007}. We also leave it to the interested readers. Here we give a regularity criterion for $\bm{b}$ in a weaker critical space $\dot{B}_{\infty,\infty}^{-1}(\R^3)$. To avoid unessential issues, we consider the problem of regularity in the whole space $\R^3$.
\begin{thm}\label{thm2}
Let $(\bm{u},\Pi)$ be a classical axisymmetric solution to Naiver--Stokes equations in $\R^3\times(-1,0)$ which blows up at time $t=0$. Then
\begin{align}\label{B}
\limsup_{t\rightarrow 0}\frac{\|\bm{b}(\cdot,t)\|_{\dot{B}_{\infty,\infty}^{-1}(\R^3)}}{\left(\ln\ln\frac{100}{-t}\right)^{\frac{1}{48}-}}=\infty,
\end{align}
where $a-$ is \crb{any} positive constant smaller than $a$. 
\end{thm}

If $\alpha=0$, we can deduce \crb{the following result}, which covers the regularity criteria \eqref{chen} and \eqref{seregin1}. It proof is directly from Lemma \ref{lemA7} and regularity criterion \eqref{seregin1}.

\begin{thm}\label{thm3}
	Assume that	$(\bm{u},\Pi)$ is an axisymmetric suitable weak solution to the Navier--Stokes equations in $Q(1)$. The solution is regular at $z_{0}$, provided that
	\begin{align}
		\limsup_{R\rightarrow0}\,G(z_{0},R) <\infty.
	\end{align}
\end{thm}

\section{Proof of Proposition \ref{lemA6}}
In this section, we prove Proposition \ref{lemA6} by a De Giorgi-Nash-Moser type method. Without loss of generality, we assume $z_{0}=(0,0)$. 

Let $B(R)=B(0,R)$, $Q(R)=Q(0,R)$ and $A(R)=A(0,R)$ be as defined in \eqref{AE.def}. Let $(\bm{u},\Pi)$ be an axisymmetric suitable weak solution to the Navier--Stokes equations \eqref{NS} in $Q(1)$. One of the important progress in \cite{Caffarelli1982} is that one-dimensional Hausdorff measure of the possible space-time singular points set for the suitable weak solution is zero. This implies that all possible singular points of the axisymmetric suitable weak solution $(\bm{u},\Pi)$ lie on the symmetry axis $r=0$, hence in the set $S=\{(x,t)\in Q(1)\,|\,r=0\}$. 

\begin{lem}[local maximum estimate]\label{lemA1}
For $0<R< 1$,
\begin{align} \label{eqA.1}
\sup_{Q(\frac12 R)\backslash S} |\Gamma| \leq N \left(\frac{1+A(R)}{R}\right)^{\frac{5}{2}}  \|\Gamma\|_{L^{2}(Q(R))}.
\end{align}
\end{lem}
\begin{proof}
Let $\frac{1}{2} \leq \sigma_{2}<\sigma_{1} \leq 1$ be arbitrary constants and $\varepsilon>0$ be a sufficient small constant. We introduce cut-off functions $\varphi(x, t)=\psi(|x|) \eta(t)$ satisfying
\begin{align*}
	\left\{\begin{array}{l}
		\operatorname{supp} \psi \subset B\left(\sigma_{1}R\right), \psi=1 \text { in } B\left(\sigma_{2}R\right), 0 \leq \psi \leq 1 , \\
		\\
		\operatorname{supp} \eta \subset\left(-\left(\sigma_{1}R\right)^{2}, 0\right], \eta=1 \text { in }\left(-\left(\sigma_{2}R\right)^{2}, 0\right], 0 \leq \eta \leq 1, \\
		\\
{\displaystyle		\left|\eta^{\prime}\right|+|\nabla^2\psi| \leq \frac{N}{\left(\sigma_{1}-\sigma_{2}\right)^2R^2},\quad \left|\frac{\nabla \psi}{\sqrt{\psi}}\right| \leq \frac{N}{\left(\sigma_{1}-\sigma_{2}\right)R},}
	\end{array}\right.
\end{align*}
and $\xi=\xi(r)$ satisfying 
$0\leq \xi(r) \leq 1$, $\xi(r)=0$ if $r\leq \varepsilon$, $\xi(r)=1$ if $r\geq 2\varepsilon$ and $|\xi^{(k)}(r)|\leq \frac{N}{\varepsilon^{k}}$ for $k=1,2$. 
	
Setp 1. Lift of regularity. We denote the truncation $\Gamma_{n}= \max(\min(\Gamma,n),-n)$ with  $n>100$.
Note that $\Gamma$ and $\bm{b}$ are smooth in $\operatorname{supp}\varphi\cap\operatorname{supp}\xi$ when $t<0$. 
Multiplying \eqref{Gamma}  by $2m|\Gamma_{n}|^{2m-2}\Gamma_{n} \varphi^{2}\xi$ with $ m\geq1$ and integrating by parts,  we have that
\begin{align}\label{eqA.2}
&\int_{B(1)}|\Gamma_{n}|^{2m}\varphi^2(t)\xi\dx+\frac{4m-2}{m}\int_{-1}^{t}\int_{B(1)}|\nabla|\Gamma_{n}|^m|^2\varphi^2\xi\dx\ds\notag\\
=&\int_{-1}^{t}\int_{B(1)}|\Gamma_{n}|^{2m}\left(\partial_{s}+\Delta+\bm{b}\cdot\nabla+\frac{2}{r}\partial_{r}\right)\left(\varphi^2\xi\right)\dx\ds\notag\\
&-2m\int_{B(1)}\left(\Gamma-\Gamma_{n}\right)|\Gamma_{n}|^{2m-2}\Gamma_{n}(t)\varphi^2\xi\dx\notag\\
&+2m\int_{-1}^{t}\int_{B(1)}\left(\Gamma-\Gamma_{n}\right)|\Gamma_{n}|^{2m-2}\Gamma_{n}\left(\partial_{s}+\Delta+\bm{b}\cdot\nabla+\frac{2}{r}\partial_{r}\right)\left(\varphi^2\xi\right)\dx\dt\notag\\
\leq&\int_{-1}^{t}\int_{B(1)}|\Gamma_{n}|^{2m}\left(\partial_{s}+\Delta+\bm{b}\cdot\nabla+\frac{2}{r}\partial_{r}\right)\left(\varphi^2\xi\right)\dx\ds\notag\\
&+2m\int_{-1}^{t}\int_{B(1)}\left(\Gamma-\Gamma_{n}\right)|\Gamma_{n}|^{2m-2}\Gamma_{n}\left(\partial_{s}+\Delta+\bm{b}\cdot\nabla+\frac{2}{r}\partial_{r}\right)\left(\varphi^2\xi\right)\dx\dt.
\end{align}

By $|\Gamma_{n}|\leq |\Gamma|\leq 2\varepsilon |\bm{u}|$ in $\operatorname{supp} \xi^{\prime}(r)$ and $\bm{u},\bm{b}\in L^{\frac{10}{3}}(Q(1))$, we can pass the limit $\varepsilon\rightarrow0$ in the inequality \eqref{eqA.2} and obtain that
\begin{align}\label{eqA.3}
&\int_{B(1)}|\Gamma_{n}|^{2m}(t)\varphi^2\dx+2\int_{-1}^{t}\int_{B(1)}|\nabla|\Gamma_{n}|^m|^2\varphi^2\dx\ds\notag\\
\leq&\int_{-1}^{t}\int_{B(1)}|\Gamma_{n}|^{2m}\left(\partial_{s}+\Delta+\bm{b}\cdot\nabla+\frac{2}{r}\partial_{r}\right)\varphi^2\dx\ds\notag\\
&+2m\int_{-1}^{t}\int_{B(1)}\left(\Gamma-\Gamma_{n}\right)|\Gamma_{n}|^{2m-2}\Gamma_{n}\left(\partial_{s}+\Delta+\bm{b}\cdot\nabla+\frac{2}{r}\partial_{r}\right)\varphi^2\dx\dt.
\end{align}
Therefore, using $\bm{b}\in L^{\frac{10}{3}}(Q(1))$, we have 
\begin{align}
	\|\Gamma_n\|_{L^{\frac{10}{3}m}(Q(\sigma_{2}R))}^{2m}\leq \frac{Nm\E(1)^{\frac{1}{2}}}{\left(\sigma_{1}-\sigma_{2}\right)^2R^2}\ 	\|\Gamma_n,\Gamma\|_{L^{\frac{20}{7}m}(Q(\sigma_{1}R))}^{2m}
\end{align}
for all $n$ and the same estimate for $\Gamma$.
By an argument of iteration, we have $\Gamma \in L^{p}(Q(R))$ for any $1\leq p<+\infty$ and $0<R<1$. 

\medskip
Step 2. Moser's iteration.
Passing the limit $n\rightarrow+\infty$ in \eqref{eqA.3}, the last integral of \eqref{eqA.3} vanishes and we obtain
\begin{align}\label{eqA.5}
	&\int_{B(1)}|\Gamma|^{2m}\varphi^2(t)\dx+2\int_{-1}^{t}\int_{B(1)}|\nabla|\Gamma|^{m}|^2\varphi^2\dx\ds\notag\\
	\leq&\int_{-1}^{t}\int_{B(1)}|\Gamma|^{2m}\left(\partial_{s}+\Delta+\bm{b}\cdot\nabla+\frac{2}{r}\partial_{r}\right)\varphi^2\dx\ds\notag\\
	\leq& \frac{N}{\left(\sigma_{1}-\sigma_{2}\right)^2R^2}\int_{-\left(\sigma_{1}R\right)^2}^{t}\int_{B(\sigma_{1}R)}|\Gamma|^{2m}\dx\ds\notag\\
	&+2\int_{-1}^{t}\int_{B(1)}\left(|\Gamma|^{m}\right)^{\frac{1}{2}}\left(|\Gamma|^{m}\varphi\right)^{\frac{3}{2}} |\bm{b}|\cdot\frac{\nabla \varphi}{\sqrt{\varphi}} \dx\ds\notag\\
	\leq& \frac{N}{\left(\sigma_{1}-\sigma_{2}\right)^2R^2}\int_{-\left(\sigma_{1}R\right)^2}^{t}\int_{B(\sigma_{1}R)}|\Gamma|^{2m}\dx\ds\notag\\
    &+\frac{N}{(\sigma_{1}-\sigma_{2})R}\int_{-1}^{t}\left\||\Gamma|^{m}\right\|_{L^{2}\left(B(\sigma_{1}R)\right)}^{\frac{1}{2}}\left\||\Gamma|^{m}\varphi\right\|_{L^{6}\left(\R^3\right)}^{\frac{3}{2}} \|\bm{b}\|_{L^{2}(B(R))}\ds\notag\\
    	\leq& \frac{N}{\left(\sigma_{1}-\sigma_{2}\right)^2R^2}\int_{-\left(\sigma_{1}R\right)^2}^{t}\int_{B(\sigma_{1}R)}|\Gamma|^{2m}\dx\ds\notag\\
    &+\frac{N \left(R A(R)\right)^{\frac{1}{2}}}{(\sigma_{1}-\sigma_{2})R}\int_{-1}^{t}\left\||\Gamma|^{m}\right\|_{L^{2}\left(B(\sigma_{1}R)\right)}^{\frac{1}{2}}\left\||\nabla\left(\Gamma|^{m}\varphi\right)\right\|_{L^{2}\left(\R^3\right)}^{\frac{3}{2}} \ds\notag\\
	\leq&\, N \frac{\left(1+A(R)\right)^2}{\left(\sigma_{1}-\sigma_{2}\right)^4 R^2}\int_{-\left(\sigma_{1}R\right)^2}^{t}\int_{B(\sigma_{1}R)}|\Gamma|^{2m}\dx\ds+\int_{-1}^{t}\int_{B(1)}|\nabla|\Gamma|^{m}|^2\varphi^2\dx\ds.
\end{align}
Accordingly, absorbing the last term in \eqref{eqA.5} by its first line, and using Lemma \ref{lemA9} on imbedding,
\begin{align}
\||\Gamma|^{m}\|_{L^{\frac{10}{3}}(Q(\sigma_{2}R))}\leq N \frac{\left(1+A(R)\right)}{\left(\sigma_{1}-\sigma_{2}\right)^2R} \||\Gamma|^{m}\|_{L^{2}(Q(\sigma_{1}R))}.
\end{align}
Picking $\mu_{j}=\frac{1}{2}+2^{-j-1}$ and $m_{j}=2\left(\frac{5}{3}\right)^{j}$ for integer $j\geq0$, we have that 
\begin{align}\label{eqA.7}
	\|\Gamma\|_{L^{m_{j+1}}(Q(\mu_{j+1}R))}\leq \left(\frac{2^{2j+2}N\left(1+A(R)\right)}{R}  \right)^{\left(\frac{3}{5}\right)^{j}} \|\Gamma\|_{L^{m_{j}}(Q(\mu_{j}R))}.
\end{align}
Thus, we have
\begin{equation}
\|\Gamma\|_{L^{m_{j}}(Q(\frac{1}{2} R))}\leq	\|\Gamma\|_{L^{m_{j}}(Q(\mu_{j}R))}
\leq N\left(\frac{\left(1+A(R)\right)}{R}\right)^{\frac{5}{2}}  \|\Gamma\|_{L^{m_{0}}(Q(\mu_{0}R))}.
\end{equation}
Note $\mu_0=1$ and $m_0=2$.
Letting $j\rightarrow\infty$, we obtain \eqref{eqA.1}.
\end{proof}

For $0<R\leq\frac{1}{4}$, we define
\begin{align}
	m_{R}=\inf_{Q(R)\backslash S}\Gamma,\quad M_{R}=\sup_{Q(R)\backslash S}\Gamma,\quad J_{R}=M_{R}-m_{R},
\end{align}
and
\begin{align}\label{eqA.10}
	h(x,t)=\left\{\begin{array}{lll}
\dis		\frac{2\left(M_{R}-\Gamma\right)}{J_{R}}, &\text{if}\ M_{R}>-m_{R},\\
		\\
\dis		\frac{2\left(\Gamma-m_{R}\right)}{J_{R}}, &\text{else}.
	\end{array}\right.
\end{align}
 Hence for $(x,t)\in Q(R)\setminus S$, $h(x,t)$ satisfies $0\leq h(x,t)\leq 2$ and
\begin{align}\label{h}
	\partial_{t} h+\left(\bm{b}\cdot\nabla\right)h-\left(\Delta-\frac{2}{r}\partial_{r}\right)h=0.
\end{align} 
For $\left(0,x_{3},t\right)\in Q(1)\backslash S$, $h(0,x_{3},t)$ equals a constant $a\geq 1$ since $\Gamma(0,x_{3},t)=0$.

\begin{lem}[initial lower bound]\label{lemA2}
There exists a constant $0<N_{4}<1$ such that
\begin{align}\label{eqA.11}
R^{-5}\|h\|_{L^{1}\left(B(\frac{1}{2} R)\times\left(-R^2,- \frac{1}{4}R^2\right)\right)}\geq N_{4} \left(1+A(R)\right)^{-1},
\end{align}
for $0<R\leq \frac{1}{4}$.
\end{lem}
\begin{proof}
We introduce the cut-off function $\widetilde{\varphi}(x, t)=\widetilde{\psi}(|x|) \widetilde{\eta}(t)$ satisfying
\begin{align*}
	\left\{\begin{array}{l}
		\operatorname{supp} \widetilde{\psi} \subset B\left(\frac{1}{2} R\right),\ \widetilde{\psi}=1 \text { in } B\left(\frac{1}{4} R\right),\ 0 \leq \widetilde{\psi} \leq 1 , \\
		\\
		\operatorname{supp} \widetilde{\eta} \subset\left(-R^{2}, -\frac{1}{4} R^2\right),\ \widetilde{\eta}=1 \text { in }\left(-\frac{7}{8} R^{2}, -\frac{3}{8} R^2\right),\ 0 \leq \widetilde{\eta} \leq 1, \\
		\\
		 \left|\nabla \widetilde{\psi}\right|\leq \frac{N}{R}, \ |\nabla^2\widetilde{\psi}| + \left|\widetilde{\eta}^{\prime}\right|\leq \frac{N}{R^2}.
	\end{array}\right.
\end{align*}
Let $\xi(x)$ be as in the proof of Lemma \ref{lemA1}.
Multiplying \eqref{h}  by $\widetilde{\varphi}\xi$ and integrating by parts, we have that
\begin{align*}
0=\int_{-R^2}^{-\frac{1}{4}R^2}\int_{B(\frac{1}{2} R)}h \left(\partial_{s}+\bm{b}\cdot\nabla\right)\left(\widetilde{\varphi} \xi\right)-\nabla h\cdot\nabla\left(\widetilde{\varphi}\xi\right)-\frac{2}{r}\partial_{r} h\cdot\left(\widetilde{\varphi}\xi\right)\dx\ds.
\end{align*}
Passing the limit $\varepsilon\rightarrow0$ using $h \in L^\infty \cap L^2H^1$, integrating by parts and using the fact that $0\leq h\leq 2$, we have
\begin{align*}
	0=\int_{-R^2}^{-\frac{1}{4}R^2}\int_{B(\frac{1}{2} R)} h \left(\partial_{s}+\Delta+\bm{b}\cdot\nabla+\frac{2}{r}\partial_{r}\right)\widetilde{\varphi} \dx\ds+4\pi a\int_{-R^2}^{-\frac{1}{4} R^2}\int_{-\frac{1}{2} R}^{\frac{1}{2} R}\widetilde{\varphi}(0,x_{3},s)\dx_{3}\ds.
\end{align*}
Thus,
\begin{align*}
	\pi R^3\leq& 4\pi a\int_{-R^2}^{-\frac{1}{4} R^2}\int_{-\frac{1}{2} R}^{\frac{1}{2} R}\widetilde{\varphi}(0,x_{3},s)\dx_{3}\ds\\
	=&-\int_{-R^2}^{-\frac{1}{4}R^2}\int_{B(\frac{1}{2} R)} h \left(\partial_{s}+\Delta+\bm{b}\cdot\nabla+\frac{2}{r}\partial_{r}\right)\widetilde{\varphi} \dx\ds\\
	\leq& NR^{-2}\|h\|_{L^{1}\left(B(\frac12 R)\times\left(-R^2,-\frac14 R^2\right)\right)}+N R^{-1}\int_{-R^2}^{-\frac{1}{4}R^2}\|h\|_{L^{2}\left(B(\frac12 R)\right)}\|\bm{b}\|_{L^{2}\left(B(R)\right)}\ds\\
	\leq &NR^{\frac{1}{2}}\|h\|_{L^{1}\left(B(\frac12 R)\times\left(-R^2,-\frac14 R^2\right)\right)}^{\frac{1}{2}}\left(1+A(R)^{\frac{1}{2}}\right).
\end{align*}
Hence \eqref{eqA.11}.
\end{proof}

Denote the cut--off function $\zeta(x)={\kappa}(|x|)$ satisfies
\begin{align}
		\operatorname{supp} {\kappa} \subset [0,1),\ \kappa=1 \text { in } [0,\frac{1}{2}],\ -N\leq {\kappa}^{\prime}\leq0 ,\ \int \zeta^2\dx=1,
\end{align}
and $\zeta_{R}(x)=R^{-\frac{3}{2}}\zeta\left(\frac{x}{R}\right)$ with $\int_{B(R)}\zeta_{R}^2\dx=1$.

\begin{lem}[weak Harnack inequality]\label{lem2.3}
	\begin{align}\label{eqA.12}
-\int_{\R^3}\ln h(x,t) \cdot\zeta_{R}^2(x)\dx\leq N  \left(1+A(R)\right)^{3},
	\end{align}
for $-\frac{1}{4} R^2\leq t<0$ with $0<R\leq \frac{1}{4}$.
\end{lem}

\begin{proof}
Denote $h_{\delta}=h+\delta$ and $H_{\delta}=-\ln\frac{h_{\delta}}{3} >0$ for $0<\delta<1$ a small constant. It is easy to see that $H_{\delta}(t)$ solves the equation
\begin{align}\label{eqA.13}
	\partial_{t} H_{\delta}+\bm{b}\cdot\nabla H_{\delta}-\left(\Delta-\frac{2}{r}\partial_{r}\right) H_{\delta}+|\nabla H_{\delta}|^2=0.
\end{align}
Let $\xi(x)$ be as in the proof of Lemma \ref{lemA1}.
Multiplying the equation \eqref{eqA.13} with $\zeta_{R}^2 \xi$ and integrating by parts, we have that for $-R^2\leq t_{0}<t<0$,
\begin{align*}
	&\int_{B(R)} H_{\delta} (t)\zeta_{R}^2\xi\dx+\int_{t_{0}}^{t}\int_{B(R)} |\nabla H_{\delta}|^2\zeta_{R}^2\xi \dx\ds\\
	=&	\int_{B(R)} H_{\delta} (t_{0})\zeta_{R}^2\xi\dx-\int_{t_{0}}^{t}\int_{B(R)}\bm{b}\cdot\nabla H_{\delta}\cdot\left(\zeta_{R}^2\xi\right)+\nabla H_{\delta}\cdot\nabla\left(\zeta_{R}^2\xi\right)+\frac{2}{r}\partial_{r} H_{\delta}\cdot\left(\zeta_{R}^2\xi\right)
	\dx\ds.
\end{align*}
Let $\overline{H}_{\delta}=\int_{B(R)}H_{\delta}\zeta_{R}^2\dx \in \mathrm{C}(-R^2,0)$. Passing $\varepsilon\rightarrow0$, we have 
\begin{align}
		&\int_{B(R)} H_{\delta} (t)\zeta_{R}^2\dx+\int_{t_{0}}^{t}\int_{B(R)} |\nabla H_{\delta}|^2\zeta_{R}^2 \dx\ds\notag\\
	=&	\int_{B(R)} H_{\delta} (t_{0})\zeta_{R}^2\dx-\int_{t_{0}}^{t}\int_{B(R)}\bm{b}\cdot\nabla H_{\delta}\cdot\zeta_{R}^2+\nabla H_{\delta}\cdot\nabla\zeta_{R}^2+\frac{2}{r}\partial_{r} H_{\delta}\cdot\zeta_{R}^2\dx\ds.
\end{align}
Since
\begin{align*}
-\int_{B(R)}\bm{b}\cdot\nabla H_{\delta}\cdot\zeta_{R}^2+\nabla H_{\delta}\cdot\nabla\zeta_{R}^2\dx\leq     NR^{-2}\left(1+A(R)\right)+\frac{1}{4}\int_{B(R)}|\nabla H_{\delta}|^2\zeta_{R}^2\dx,
\end{align*}
and by $H_\delta|_{r=0} \le \ln 3$ and Lemma \ref{lemA8} (weighted Poincar\'e inequality),
\begin{align*}
&\int_{B(R)}-\frac{2}{r}\partial_{r} H_{\delta}\cdot\zeta_{R}^2\dx=\int_{B(R)}-\frac{2}{r}\partial_{r} \left(H_{\delta}-\overline{H}_{\delta}\right)\cdot\zeta_{R}^2\dx\\ 
&\quad=4\pi\int_{-R}^{R} \left(H_{\delta}-\overline{H}_{\delta}\right)\zeta_{R}^2\left|_{r=0}\right.\dx_{3}+2\int_{B(R)} \left(H_{\delta}-\overline{H}_{\delta}\right)\, \frac{2}{r}\partial_{r}\zeta_{R}^2\dx  \\
&\quad \leq  NR^{-2}-\overline{H}_{\delta}R^{-2}
+ NR^{-2}
\left(\int_{B(R)}|H_{\delta}-\overline{H}_{\delta}|^2\zeta_{R}^2\dx\right)^{\frac{1}{2}}\\
&\quad \leq NR^{-2}-\overline{H}_{\delta}R^{-2}+NR^{-1}\left(\int_{B(R)}|\nabla H_{\delta}|^2\zeta_{R}^2\dx\right)^{\frac{1}{2}}\\
&\quad \leq NR^{-2}-\overline{H}_{\delta}R^{-2}+\frac{1}{4}\int_{B(R)}|\nabla H_{\delta}|^2\zeta_{R}^2\dx,
\end{align*}
we have that
\begin{align}
\overline{H}_{\delta}(t)
\leq&\overline{H}_{\delta}(t_{0})+R^{-2}\int_{t_{0}}^{t}\left(N\  \left(1+A(R)\right)-\overline{H}_{\delta}\right)\ds\notag\\
&-\frac{1}{2}\int_{t_{0}}^{t}\int_{B(R)} |\nabla H_{\delta}|^2\zeta_{R}^2 \dx\ds,
\end{align}
which implies that
\begin{align}\label{eqA.20}
	\overline{H}_{\delta}(t)\leq \overline{H}_{\delta}(t_{0})+N \left(1+A(R)\right).
\end{align}
Applying Lemma \ref{lemA3} (Nash inequality) with $f=\frac{h_{\delta}}3$, $d\mu=\zeta_{R}^2\dx$ and Lemma \ref{lemA8} (weighted Poincar\'e inequality), one has
\begin{align}
\left|\ln \int_{B(R)}\frac{h_{\delta}}3 \ \zeta_{R}^2\dx+\overline{H}_{\delta}\right|^2 \left(\int_{B(R)}h_{\delta} \ \zeta_{R}^2\dx\right)^2
\leq& N \int_{B(R)}|H_{\delta}-\overline{H}_{\delta}|^2\zeta_{R}^2\dx\notag\\
\leq& N R^2 \int_{B(R)}|\nabla H_{\delta}|^2\zeta_{R}^2\dx.
\end{align}
Thereby we obtain that 
\begin{align}
\notag
\overline{H}_{\delta}(t)
\leq&\overline{H}_{\delta}(t_{0})+R^{-2}\int_{t_{0}}^{t}\left(N \left(1+A(R)\right)-\overline{H}_{\delta}\right)\ds\\
&-NR^{-2}\int_{t_{0}}^{t}\left|\ln \int_{B(R)}\frac{h_{\delta}}3 \ \zeta_{R}^2\dx+\overline{H}_{\delta}\right|^2 \left(\int_{B(R)}h_{\delta} \ \zeta_{R}^2\dx\right)^2\ds. \label{eq2.20}
\end{align}

Let $\chi(s)$ be the characteristic function of the set
\begin{align*}
W=\left\{s\in(-R^2,-\frac{1}{4} R^2):\|h_{\delta}\|_{L^{1}(B(\frac{1}{2} R))}\geq N_{4} R^3\left(1+A(R)\right)^{-1}\right\},
\end{align*}
where $N_4$ is the constant in Lemma \ref{lemA2}.
We assert that $|W|> \frac{N_{4} \left(1+A(R)\right)^{-1}}{8} R^2$. In fact, if $|W|\leq \frac{N_{4}\left(1+A(R)\right)^{-1}}{8} R^2$, then
\begin{align*}
	\|h_{\delta}\|_{L^{1}\left(B(\frac12 R)\times\left(-R^2,-\frac{1}{4} R^2\right)\right)}\leq& \int_{W}\|h_{\delta}\|_{L^{1}(B(\frac12 R))}\ds+\int_{(-R^2,-\frac{1}{4} R^2)\backslash W}\|h_{\delta}\|_{L^{1}(B(\frac12 R))}\ds\\
	\leq& \frac{\pi}{6} R^3 |W| {\cdot\norm{h_\delta}_\infty} +\frac{3N_{4} R^5 \left(1+A(R)\right)^{-1}}{4}\\
	\leq &  (\tfrac\pi{16}+\tfrac{3}{4}) N_{4}R^5 \left(1+A(R)\right)^{-1},
\end{align*}
which contradicts Lemma \ref{lemA2}. Thus, one has by \eqref{eq2.20} that for $-R^2\leq t_{0}<t<0$,
\begin{align}
\notag
	\overline{H}_{\delta}(t)\leq& \overline{H}_{\delta}(t_{0})+R^{-2}\int_{t_{0}}^{t}\left(N_{5} \left(1+A(R)\right)-\overline{H}_{\delta}\right)\ds\\
	&-N_{6} {R^{-2}}\left(1+A(R)\right)^{-2}\int_{t_{0}}^{t}\left|\ln \int_{B(R)}{\frac {h_{\delta}}3} \ \zeta_{R}^2\dx+\overline{H}_{\delta}\right|^2 \chi(s)\ds.\label{eq2.21}
\end{align}

We claim that  for $-\frac14 R^2\leq t<0$,
\begin{align}\label{eqA.21}
	\overline{H}_{\delta}(t)\leq N \left(1+A(R)\right)^3,
\end{align}
which implies \eqref{eqA.12} directly by passing $\delta\rightarrow0$.

If for some $t_{0}\in [-R^2,-\frac14 R^2) $ we have
\begin{align*}
\overline{H}_{\delta}(t_{0})\leq 2\ln \frac{1+A(R)}{N_{4}}+N_{5}\left(1+A(R)\right)+100,
\end{align*}
then by \eqref{eqA.20}, we have \eqref{eqA.21}. Otherwise, we have
\begin{align*}
	\overline{H}_{\delta}(s)> 2\ln \frac{1+A(R)}{N_{4}}+N_{5}\left(1+A(R)\right)+100,
\end{align*}
for all $s\in [-R^2,-\frac14 R^2) $. For $s\in W\cap (-R^2,-\frac14 R^2)$,
\begin{align*}
\ln \int_{B(R)}\frac{h_{\delta}}{3} \ \zeta_{R}^2\dx
\ge \ln  \int_{B(\frac12 R)}\frac {h_{\delta}}{3R^3} \ \dx \ge \ln \left( \frac {N_4} {3(1+A(R))} \right)
\geq -\frac{1}{2} \overline{H}_{\delta}.
\end{align*}
Therefore, for $-R^2\leq t_{0}<t\leq-\frac14 R^2$, \eqref{eq2.21} gives (noting its first integral is nonpositive)
\begin{align*}
	\overline{H}_{\delta}(t)
		\leq& \overline{H}_{\delta}(t_{0})-\frac{1}{4}N_{6} R^{-2}\left(1+A(R)\right)^{-2}\int_{t_{0}}^{t}\overline{H}_{\delta}^2\  \chi(s)\ds.
\end{align*}
By comparison with the solution of $g(t) = g(t_0)- C \int_{t_0}^t g^2 \chi(s)\ds$, $g(t)^{-1} = g(t_0)^{-1} + 
C \int_{t_0}^t \chi(s)\ds \ge C \int_{t_0}^t \chi(s)\ds$, we get
\begin{align}
	\overline{H}_{\delta}(-\frac14 R^2)\leq& \frac{4}{N_{6} R^{-2}\left(1+A(R)\right)^{-2}\int_{-R^2}^{-\frac14 R^2}\chi(s)\ds}\notag\\
	\leq& N \left(1+A(R)\right)^{3}.
\end{align}
By \eqref{eqA.20} with $t_0 = -\frac14 R^2$, we have \eqref{eqA.21}.
\end{proof}

\begin{lem}[strong Harnack inequality] \label{lemA.5}
Let $0<\beta<\frac 18$, $0<\tau<1$, and $\omega(R)=\left(\ln \ln\frac{100}{R}\right)^{-1}$.
If $A(R)$ satisfies that for all $0<R\leq\frac{1}{4}$,
\begin{align}
	A(R) \omega(R)^{\beta}\leq K,
\end{align} 
then for $0<R\leq \frac{1}{4}$,
\begin{align}\label{eqA.24}
	\inf_{Q(\frac14 R)\backslash S} h\geq \frac{1}{2}  \lambda(R),
\end{align}
where \crb{$\lambda(R)=N_{7} \left(\ln\frac{100}{R}\right)^{\tau-1}$ and $0<N_{7}=N_7(\tau)<1$ is a sufficiently small constant.}
\end{lem}

\begin{proof}
By \eqref{eqA.12} of Lemma \ref{lem2.3}, one has that for $-\frac14 R^2\leq t<0$
\begin{align*}
-\int_{h\leq \lambda(R)}\ln h \cdot\zeta_{R}^2\dx\leq N   \left(1+A(R)\right)^{3},
\end{align*}
which implies that 
\begin{equation}\label{eq2.26}
\left|\left\{x\in B(\frac{1}{2} R) \ \bigg|\ h\leq \lambda(R)\right\}\right|\leq \frac{N  R^3 \left(1+A(R)\right)^{3}}{-\ln \lambda(R)}.
\end{equation}
By a similar argument as Step 2 in the proof of Lemma \ref{lemA1},  using $(\lambda(R)-h)_{+}=0$ on $r=0$,
we have the same conclusion of Lemma \ref{lemA1} for $\Gamma= (\lambda(R)-h)_{+}$ that
\[
\sup_{Q(\frac14 R)\backslash S}\left(\lambda(R)-h\right)_{+}\leq N   \left(\frac{1+A(R)}{R}\right)^{\frac{5}{2}}  \|\left(\lambda(R)-h\right)_{+}\|_{L^{2}(Q(\frac12 R))}.
\]
By \eqref{eq2.26},
\begin{align*}
\sup_{Q(\frac14 R)\backslash S}\left(\lambda(R)-h\right)_{+}
\leq & N \lambda(R)\frac{ \left(1+A(R)\right)^{4}}{\sqrt{-\ln \lambda(R)}}\\
\leq & \frac{NK^{4}}{ \left(-\ln N_{7}\right)^{\frac{1}{2}-4\beta}\crb{(1-\tau)^{4\beta}}}  \lambda(R),
\end{align*}
which implies that
\begin{align*}
\inf_{Q(\frac14 R)\backslash S} h\geq \lambda(R) \left(1-\frac{NK^{4}}{ \left(-\ln N_{7}\right)^{\frac{1}{2}-4\beta}\crb{(1-\tau)^{4\beta}}} \right).
\end{align*}
Thus, we obtain \eqref{eqA.24} if we pick $N_{7}$ sufficiently small such that
\begin{equation}
1-\frac{NK^{4}}{ \left(-\ln N_{7}\right)^{\frac{1}{2}-4\beta}\crb{(1-\tau)^{4\beta}}}\geq \frac{1}{2} . \qedhere
\end{equation}
\end{proof}

Now, we are in the position to prove Proposition \ref{lemA6}. 

\begin{proof}[Proof of Proposition \ref{lemA6}]
Noting that $m_{R}\leq m_{\frac14 R}\leq M_{\frac14 R}\leq M_{R}$, \eqref{eqA.10} and \eqref{eqA.24}, we obtain that
\begin{align}
J_{\frac14 R}\leq \left(1-\frac{\lambda(R)}{4}\right) J_{R},
\end{align}
for $0<R\leq \frac{1}{4}$. By standard iteration and using $\ln(1-t) < -t$ for $0<t<1$, we have that for $j\geq 1$,
\begin{align*}
	J_{2^{-2j} R}\leq& \exp\left(\Sigma_{k=0}^{j-1}\ln\left(1-\frac{N_{7}}{4}\left(\ln\frac{100\cdot 2^{2k} }{R}\right)^{\tau-1}\right)\right) J_{R}\\
	\leq&\exp\left(-\frac{N_{7}}{4}\Sigma_{k=0}^{j-1}\left(\ln\frac{100\cdot 2^{2k} }{R}\right)^{\tau-1}\right)J_{R}\\
	\leq&\exp\left(-\frac{N_{7}}{4}\int_{0}^{j}\left(\ln\frac{100}{R}+s\ln4 \right)^{\tau-1}\ds\right)J_{R}\\
	\leq &\, e^{-\frac{N_{7}}{4\tau\ln 4 }\left(\left(\ln\frac{100}{R}+j \ln4\right)^{\tau}-\left(\ln\frac{100}{R}\right)^{\tau}\right)} J_{R},
\end{align*}
which implies that for $0<\rho<R\leq \frac{1}{4}$,
\begin{align}
	J_{\rho}\leq& e^{-\frac{N_{7}}{{4\tau\ln 4}}\left(\left(\ln\frac{100}{\rho}-\ln4\right)^{\tau}-\left(\ln\frac{100}{R}\right)^{\tau}\right)} J_{R}\notag\\
	\leq& e^{-\frac{N_{7}}{{4\tau\ln 4}}\left(\left(\ln\frac{100}{\rho}\right)^{\tau}-\left(\ln \frac{100}{R}\right)^{\tau}-2\right)} J_{R}.
\end{align}
Hence \eqref{OSC}. 

Finally, if \eqref{A} holds for \emph{all} $z_{0}=(0,x_{0,3},t_{0})\in Q(\frac{1}{8})$, then we have \eqref{OSC} for all such $z_0$. Being suitable in $Q(1)$, we have $A(1)+E(1)<\infty$. By Lemma \ref{lemA1}, $\norm{\Gamma}_{L^\infty(Q(3/8))} \lec  (1+A(1))^{3}$. 
We can deduce \eqref{lnr} from \eqref{OSC} and the fact that $\Gamma (0,x_{3},t)=0$ for a.e.~$(0,x_{3},t)\in Q(\frac{1}{8})$.
\end{proof}

\section{Local energy estimates}
In this section, we will give some useful local energy estimates. Let $(\bm{u},\Pi)$ be a suitable weak solution of \eqref{NS} in $Q(1)$. Recall the quantities $A,E,C,D,\E$ defined in \eqref{AE.def} and $G$ and $G_\alpha$ in \eqref{GR.def}.

\begin{lem}\label{lem1}
	 For $z_{0}=({x_{0},t_{0}})\in Q(\frac{1}{8})$ and $0<\rho\leq \frac{1}{4}$, we have
	\begin{align}
		A\left(z_{0},\rho\right)+E\left(z_{0},\rho\right)\leq N \left(1+C(z_{0},2\rho)+D(z_{0},2\rho)\right).
	\end{align}
\end{lem}
\begin{proof}
	By choosing a suitable test function $\varphi$ in the local energy inequality \eqref{local energy}, we get
	\begin{align*}
		A(z_{0},\rho)+E(z_{0},\rho) \leq& N\left(C(z_{0},2 \rho)^{\frac{2}{3}}+C(z_{0},2 \rho)+\frac{1}{\rho^{2}}\|\bm{u}\|_{L^{3}\left(Q(z_{0},2 \rho)\right)}\|\Pi\|_{L^{\frac{3}{2}}\left(Q(z_{0},2 \rho)\right)}\right)\\
		\leq&N\left(C(z_{0},2\rho)^{\frac{2}{3}}+C(z_{0},2\rho)+C(z_{0},2\rho)^{\frac{1}{3}}D(z_{0},2\rho)^{\frac{2}{3}}\right)\\
		\leq &N \left(1+C(z_{0},2\rho)+D(z_{0},2\rho)\right).\qedhere
	\end{align*}
\end{proof}
\begin{lem}\label{lem2}
	For $0<\rho\leq R\leq \frac{1}{4}$, we have
	\begin{align}\label{eq2.2}
		C(z_{0},\rho)\leq N\left(\frac{R}{\rho}\right)^2 \E(z_{0},R)^{1-\frac{\gamma}{6}} G(z_{0},R)^{1+\frac{\gamma}{3}}+N\E(1)^{\frac{9}{2}}\left(\frac{R}{\rho}\right)^{2}E(z_{0},R)^{\frac{3}{4}}.
	\end{align}
\end{lem}
\begin{proof}
	Set $\frac{1}{p_{1}}=\frac{3-\frac{\gamma}{2}-\frac{\gamma^2}{5}-\frac{3+\gamma}{p}}{6-\gamma}$ and $\frac{1}{q_{1}}=\frac{3-\frac{\gamma}{2}-\frac{\gamma^2}{5}-\frac{3+\gamma}{q}}{6-\gamma}$. Thus, we have that $\frac{3}{p_{1}}+\frac{2}{q_{1}}=\frac{3}{2}$  with  $2\leq p_{1}\leq 6$. 
	Using Lemma \ref{lemA1}, Lemma \ref{lemA9} and the H\"{o}lder inequality, we have that
	\begin{align*}
		C(z_{0},\rho)\leq & N\rho^{-2}\iint_{Q(z_{0},R)}|\bm{b}|^3\dx\dt+N\rho^{-2}\iint_{Q(z_{0},R)}|u^{\theta}|^3\dx\dt\\
		\leq &N\rho^{-2}R^{\frac{5\gamma}{6}+\frac{\gamma^2}{3}}\|\bm{b}\|_{L^{p_{1},q_{1}}(Q(z_{0},R))}^{2-\frac{\gamma}{3}}\|\bm{b}\|_{L^{p,q}(Q(z_{0},R))}^{1+\frac{\gamma}{3}}+N\E(1)^{\frac{9}{2}}\rho^{-2}\iint_{Q(z_{0},R)}\left|\frac{u^{\theta}}{r}\right|^{\frac{3}{2}}\dx\dt\\
		\leq& N\left(\frac{R}{\rho}\right)^2 \E(z_{0},R)^{1-\frac{\gamma}{6}} G(z_{0},R)^{1+\frac{\gamma}{3}}+N\E(1)^{\frac{9}{2}}\left(\frac{R}{\rho}\right)^2 E(z_{0},R)^{\frac{3}{4}}.\qedhere
	\end{align*}
\end{proof}
\begin{lem}\label{lem2'}
	Denote $B(t_{0},R)=R^{-\frac{1}{3}}\|\bm{b}\|_{L^{6}(t_{0}-R^2,t_{0};\dot{B}_{\infty,\infty}^{-1}(\R^3))}$.
	For $0<2\rho\leq R\leq \frac{1}{4}$, we have
	\begin{align}\label{eq3.3}
	C(z_{0},\rho)\leq N\left(\frac{R}{\rho}\right)^2 \left( B(t_{0},R)^{\frac{3}{2}}+\E(1)^{\frac{9}{2}}\right) \E(z_{0},R)^{\frac{3}{4}}.
	\end{align}
\end{lem}
\begin{proof}
By \cite[Theorem 2.42]{Bahouri2011} and \cite[Lemma 2.2]{Seregin2018}, we have that
\begin{align}\label{eq3.4}
	\|\bm{b}\|_{L^{4}(B(x_{0},\frac{1}{2}R))}\leq N\, \|\bm{b}\|_{\dot{B}_{\infty,\infty}^{-1}(\R^3)}^{\frac{1}{2}}\left(\|\nabla\bm{b}\|_{L^{2}(B(x_{0},R))}+R^{-1}\|\bm{b}\|_{L^{2}(B(x_{0},R))}\right)^{\frac{1}{2}}.
\end{align}
Therefore, by Lemma \ref{lemA1} and \eqref{eq3.4}, we have
\begin{align*}
C(z_{0},\rho)\leq & N\rho^{-2}\iint_{Q(z_{0},\frac{1}{2}R)}|\bm{b}|^3\dx\dt+N\rho^{-2}\iint_{Q(z_{0},R)}|u^{\theta}|^3\dx\dt\\
\leq&N\rho^{-2}R^{\frac{3}{4}}\int_{t_{0}-R^2}^{t_{0}} \|\bm{b}\|_{L^{4}(B(x_{0},\frac{1}{2}R))}^{3}\ds+N\E(1)^{\frac{9}{2}}\rho^{-2}\iint_{Q(z_{0},R)}\left|\frac{u^{\theta}}{r}\right|^{\frac{3}{2}}\dx\dt\\
\leq&N\rho^{-2}R^{\frac{3}{4}}\int_{t_{0}-R^2}^{t_{0}} \|\bm{b}\|_{\dot{B}_{\infty,\infty}^{-1}(\R^3)}^{\frac{3}{2}}\left(\|\nabla\bm{b}\|_{L^{2}(B(x_{0},R))}+R^{-1}\|\bm{b}\|_{L^{2}(B(x_{0},R))}\right)^{\frac{3}{2}}\ds\\
&+N\E(1)^{\frac{9}{2}}\left(\frac{R}{\rho}\right)^2 E(z_{0},R)^{\frac{3}{4}}\\
\leq &N\left(\frac{R}{\rho}\right)^2 \E(z_{0},R)^{\frac{3}{4}} B(t_{0},R)^{\frac{3}{2}}+N\E(1)^{\frac{9}{2}}\left(\frac{R}{\rho}\right)^{2}E(z_{0},R)^{\frac{3}{4}}.
\end{align*}
Hence \eqref{eq3.3}.
\end{proof}
\begin{lem}\cite[Lemma 3.4]{Gustafson2007}\label{lem3}
	For $0<2\rho\leq R\leq\frac{1}{4}$, we have
	\begin{align}
		D(z_{0},\rho)\leq N\left(\frac{\rho}{R}\right) D(z_{0},R)+N\left(\frac{R}{\rho}\right)^2 C(z_{0},R).
	\end{align}
\end{lem}

\section{Proof of main theorems}
\begin{lem}\label{lemA7}
	Assume that	$(\bm{u},\Pi)$ is a suitable weak solution to the Navier--Stokes equations in $Q(1)$, which satisfies that for all $0<R\leq \frac{1}{4}$,
	\begin{align}
 G_{\alpha}(z_{0},R) =G(z_{0},R) \omega(R)^{\alpha}\leq G.
	\end{align}
	There exists a constant $N_{0}$ such that for all $0<R \leq \frac{1}{4}$,
	\begin{align}\label{E}
		\E_{\beta}(z_{0},R)=\E(z_{0},R)\omega(R)^{\beta}\leq N_{0}^{\frac{12}{\gamma}} \left(1+ \E(1)^{18}+ G^{\frac{6+2\gamma}{\gamma}}\right),
	\end{align}
	where $\alpha=\frac{\gamma}{6+2\gamma}\beta$.
\end{lem}
\begin{proof} 
	By Lemma \ref{lem1}, Lemma \ref{lem2} and Lemma \ref{lem3}, we have that for $0<4\rho\leq R\leq \frac{1}{4}$,
	\begin{align*}
		\E(z_{0},\rho) \leq& N \left(1+C(z_{0},2\rho)+D(z_{0},2\rho)\right)+D(z_{0},\rho)\\
		\leq&N+N\left(\frac{R}{\rho}\right)^2\left( \E(z_{0},R)^{1-\frac{\gamma}{6}} G(z_{0},R)^{1+\frac{\gamma}{3}}+\E(1)^{\frac{9}{2}}E(z_{0},R)^{\frac{3}{4}}\right)+N\left(\frac{\rho}{R}\right) \E(z_{0},R).
	\end{align*}
	Since $\omega(R)$ is non-decreasing for $0<R\leq \frac{1}{4}$, we have that
	\begin{align}
		\E_{\beta}(z_{0},\rho)\leq& N\left(1+\left(\frac{R}{\rho}\right)^8 \E(1)^{18}\right)+N\left(\frac{R}{\rho}\right)^2 \E_{\beta}(z_{0},R)^{1-\frac{\gamma}{6}} G_{\alpha}(z_{0},R)^{1+\frac{\gamma}{3}}\notag\\
		&+\left(\frac{1}{8}+N\left(\frac{\rho}{R}\right)\right) \E_{\beta}(z_{0},R)\notag\\
		\leq& N_{1}\left(1+\left(\frac{R}{\rho}\right)^8 \E(1)^{18}+ \left(\frac{R}{\rho}\right)^{\frac{12}{\gamma}} G^{\frac{6+2\gamma}{\gamma}}\right)+\left(\frac{1}{4}+N_{1}\frac{\rho}{R}\right)\E_{\beta}(z_{0},R).
	\end{align}
	We can pick a sufficient small constant $0<\vartheta \leq \frac{1}{4}$ such that
	\begin{align*}
		\frac{1}{4}+N_{1} \vartheta\leq \frac{1}{2}.
	\end{align*}
	Thus, for any $0<R \leq\frac{1}{4}$, we have
	\begin{align*}
		\E_{\beta}(z_{0},\vartheta R) \leq \frac{1}{2} \E_{\beta}(R)+N_{2},
	\end{align*}
	where  $N_{2}=N_{1}\left(1+\vartheta^{-8} \E(1)^{18}+ \vartheta^{-\frac{12}{\gamma}} G^{\frac{6+2\gamma}{\gamma}}\right)$. By standard iterations, we have that for all $0<R\leq \frac{1}{4}$,
	\begin{align}
		\E_{\beta}(z_{0},R)\leq 100\,\vartheta^{-2} \  \E(1)+ 2N_{2}.
	\end{align}
	Hence \eqref{E}. 
\end{proof}
\begin{lem}\label{lemA10}
		Assume that $(\bm{u},\Pi)$ is a suitable weak solution to the Navier--Stokes equations in $\R^3\times(-1,0)$. If there exists a constant $B$ such that for all $0<R\leq \frac{1}{4}$,
	\begin{align}
	B(t_{0},R)\omega(R)^{\frac{\beta}{6}} \leq B,
	\end{align}
	then for all $0<R \leq \frac{1}{4}$ and $x_{0}\in \R^3$,
	\begin{align}
		\E_{\beta}(z_{0},R)=\E(z_{0},R)\omega(R)^{\beta}\leq N \left(1+ \E(1)^{18}+B^{6}\right).
	\end{align}
\end{lem}
\begin{proof}
	By Lemma \ref{lem1}, Lemma \ref{lem2'} and Lemma \ref{lem3}, we have that for $0<8\rho\leq R\leq \frac{1}{4}$,
		\begin{align*}
		\E(z_{0},\rho) \leq& N \left(1+C(z_{0},2\rho)+D(z_{0},2\rho)\right)+D(z_{0},\rho)\\
		\leq&N+N\left(\frac{R}{\rho}\right)^2\left( B(t_{0},R)^{\frac{3}{2}}+\E(1)^{\frac{9}{2}}\right) \E(z_{0},R)^{\frac{3}{4}}+N\left(\frac{\rho}{R}\right) \E(z_{0},R).
	\end{align*}
	Since $\omega(R)$ is non-decreasing for $0<R\leq \frac{1}{4}$, we have that
	\begin{align}
		\E_{\beta}(z_{0},\rho)\leq& N \left(\frac{R}{\rho}\right)^8\left(1+ \E(1)^{18}+B^{6}\right)+\left(\frac{1}{4}+N\frac{\rho}{R}\right)\E_{\beta}(z_{0},R).
	\end{align}
The rest is analogous to the proof of Lemma \ref{lemA7} and we omit the details.
\end{proof}
Now we are in a position to prove Theorems \ref{thm1} and \ref{thm2}.

By \eqref{G}, Lemma \ref{lemA7} and then  Proposition \ref{lemA6}, we obtain that for $0<r\leq \frac{1}{4}$,
\begin{align}\label{lnr2}
|\Gamma(r,x_{3},t)|\leq N |\ln r|^{-2},
\end{align}
which implies Theorem \ref{thm1} by a similar argument in \cite{Lei2017}.

Theorem \ref{thm2} is proved by reductio ad absurdum.
We assume that for some $0<\beta<\frac{1}{8}$,
\begin{align}
\|\bm{b}(\cdot,t)\|_{\dot{B}_{\infty,\infty}^{-1}(\R^3)}\leq N \left(\ln\ln\frac{100}{-t}\right)^{\frac{\beta}{6}},\quad -\frac{1}{2}<t<0.
\end{align}
Notice that for $-\frac{1}{2}<t_{0}<0$ and $0<R\leq \frac{1}{4}$,
\begin{align*}
B(t_{0},R)\omega(R)^{\frac{\beta}{6}}\leq & N R^{-\frac{1}{3}}\left(\int_{t_{0}-R^2}^{t_{0}} \left(\ln\ln\frac{100}{-t}\right)^{\beta}\dt\right)^{\frac{1}{6}}\omega(R)^{\frac{\beta}{6}}\\
\leq&N \left(\int_{0}^{1} \left(\ln\ln\frac{100}{-t_{0}+R^2 s}\right)^{\beta}\ds\right)^{\frac{1}{6}}\omega(R)^{\frac{\beta}{6}}\\
\leq & N\left(\int_{0}^{1} \left(\ln\ln\frac{100}{ s}+\ln\ln\frac{1}{R^2}\right)^{\beta}\ds\right)^{\frac{1}{6}}\omega(R)^{\frac{\beta}{6}}\\
\leq& N.
\end{align*}
Therefore, by Lemma \ref{lemA10} and then  Proposition \ref{lemA6}, we obtain \eqref{lnr2},
which implies  that the solution is regular in $\R^3\times(-\frac{1}{2},0]$. It leads to a
contradiction. Hence, Theorem \ref{thm2}.
\appendix
\section{}
\begin{lem}\label{lemA9} For $\frac{3}{p}+\frac{2}{q}=\frac{3}{2}$ with $2\leq p \leq 6$, there exists a constant $N$ such that
	\begin{align}\label{A100}
		\|f\|_{L^{p,q}(Q(z_{0},R))}\leq N \left(\|f\|_{L^{2,\infty}(Q(z_{0},R))}+\|\nabla f\|_{L^{2,2}(Q(z_{0},R))}\right).
	\end{align}
\end{lem}
The space $L^{p,q}$ is defined above \eqref{AE.def}.

\begin{proof}
Without loss of generality, we may assume $z_{0}=(0,0)$. Define the integral mean value $\left(f\right)_{B(R)}=\dashint_{B(R)}f(y,t)\dy$. By H\"{o}lder inequality and Poincar\'{e} inequality, we have
\begin{align*}
\|f\|_{L^{p}(B(R))}\leq& \|f-\left(f\right)_{B(R)}\|_{L^{p}(B(R))}+\|\left(f\right)_{B(R)}\|_{L^{p}(B(R))}\\
\leq &N \|f\|_{L^{2}(B(R))}^{\frac{3}{p}-\frac{1}{2}} \|\nabla f\|_{L^{2}(B(R))}^{\frac{3}{2}-\frac{3}{p}}+N R^{\frac{3}{p}-\frac{3}{2}} \|f\|_{L^{2}(B(R))}.
\end{align*}
Therefore
\begin{align*}
\|f\|_{L^{p,q}(Q(R))}\leq N\|f\|_{L^{2,\infty}(Q(R))}^{\frac{3}{p}-\frac{1}{2}} \|\nabla f\|_{L^{2,2}(Q(R))}^{\frac{3}{2}-\frac{3}{p}}+N  \|f\|_{L^{2,\infty}(Q(R))}.
\end{align*}
Hence \eqref{A100}.
\end{proof}
\begin{lem} [Nash inequality, see Section 5.3 in \cite{Chen2009}]\label{lemA3}
Let $M \geq 1$ be a constant and $\mu$ be a probability measure. Then for all $0 \leq f \leq M$, there holds
	\begin{align*}
		\left|\ln \int f d \mu-\int \ln f d \mu\right| \leq \frac{M\|g\|_{L^{2}\left(d\mu\right)}}{\int f d \mu},
	\end{align*}
	where $g=\ln f-\int \ln f d \mu$.
\end{lem}

\begin{lem}[weighted Poincar\'{e} inequality]\label{lemA8}
In $B(R) \subset \R^3$, let the weight $\Lambda_{R}(x)=R^{-3}\Lambda(\frac{x}{R})$ where $\Lambda(x)=\Lambda(|x|)$ is non-increasing function with  $0\leq \Lambda \leq 1$ and $\int_{B(1)}\Lambda\dx=1$. For $1\leq p<\infty$ and $f\in W^{1,p}(B(R))$, we have
\begin{align*}
	\int_{B(R)}\left|f-\int_{B(R)}f\cdot \Lambda_{R}\dy\right|^p \cdot\Lambda_{R}\dx\leq 2^{p+6}\cdot R^p \int_{B(R)}\left|\nabla f\right|^p\cdot\Lambda_{R}\dx.
\end{align*}
\end{lem}

\textit{Remark.} By the same proof, this lemma remains valid if we relax the condition on $\Lambda$ to $0 \le \Lambda \le C_1$, $\int_{B(1)}\Lambda=1$, and $$\min_{0 \le t \le 1} \Lambda(x+t(y-x)) \ge C_2 \min(\Lambda(x),\Lambda(y)),\quad \forall x,y \in B(1),$$ with the constant $2^{p+6}$ multiplied by $C_1C_2$. It needs not be monotone or radial.

\begin{proof} Define $g(x)= f(R x)$. It is sufficient to prove that
\begin{align}\label{AA1}
\int_{B(1)}\left|g-\int_{B(1)}g\cdot \Lambda\dy\right|^{p}\cdot \Lambda\dx\leq 2^{p+6} \int_{B(1)}|\nabla g|^{p} \cdot \Lambda\dx.
\end{align}
We first show that for $y\in B(1)$,
\begin{align}\label{AA3}
\int_{B(1)}\left|g(x)-g(y)\right|^p \cdot\Lambda(x)\Lambda(y)\dx\leq  \frac{2^{p+2}}{p+2} \int_{B(1)}|\nabla g(z)|^p\cdot |z-y|^{-2} \cdot\Lambda(z)\, \mathrm{d} z.
\end{align}
Since
\begin{align*}
	\left|g(x)-g(y)\right|\leq \int_{0}^{1}|\nabla g(y+t(x-y)) |\cdot |x-y|\dt,
\end{align*}
and $\Lambda(x)\Lambda(y)\le \min(\Lambda(x),\Lambda(y))\cdot \max \Lambda \le  \Lambda (y+t(x-y))\cdot1$,
we have
\begin{align*}
&\int_{B(1)}\left|g(x)-g(y)\right|^p \cdot\Lambda(x)\Lambda(y)\dx\\
 \leq &  \int_{0}^{1}\int_{B(1)}|\nabla g(y+t(x-y))|^p\cdot |x-y|^p \cdot \Lambda(y+t(x-y)) \dx\dt\\
 \leq&  \int_{0}^{1}\int_{0}^{2}\int_{B(1)\cap \partial B(y,s)}|\nabla g(y+t(x-y))|^p\cdot |x-y|^p \cdot\Lambda(y+t(x-y)) \,\mathrm{d} S(x)\ds\dt\\
 =& \int_{0}^{1}\int_{0}^{2}s^p\int_{B(1)\cap \partial B(y,s)}|\nabla g(y+t(x-y))|^p\cdot\Lambda(y+t(x-y))\, \mathrm{d} S(x)\ds\dt.
\end{align*}
Letting $z=y+t(x-y)$, it is bounded by
\begin{align*}
\leq & \int_{0}^{1}\int_{0}^{2} s^p t^{-2}\int_{B(1)\cap \partial B(y,ts)}|\nabla g(z)|^p \cdot\Lambda(z)\, \mathrm{d} S(z)\ds\dt\\
= & \int_{0}^{1}\int_{0}^{2} s^{p+2}\int_{B(1)\cap \partial B(y,ts)}|\nabla g(z)|^p\cdot |z-y|^{-2} \cdot\Lambda(z)\, \mathrm{d} S(z)\ds\dt\\
= & \int_{0}^{2} s^{p+2}\int_{0}^{1}\int_{B(1)\cap \partial B(y,ts)}|\nabla g(z)|^p\cdot |z-y|^{-2} \cdot\Lambda(z)\, \mathrm{d} S(z)\dt\ds\\
=& \int_{0}^{2} s^{p+1}\int_{B(1)\cap  B(y,s)}|\nabla g(z)|^p \cdot|z-y|^{-2} \cdot\Lambda(z)\, \mathrm{d} z\ds\\
\leq&\,  \frac{2^{p+2}}{p+2} \int_{B(1)}|\nabla g(z)|^p\cdot |z-y|^{-2} \cdot\Lambda(z)\, \mathrm{d} z.
\end{align*}

Now we proceed to prove \eqref{AA1}. By \eqref{AA3}, we have
\begin{align*}
\int_{B(1)}\left|g-\int_{B(1)}g\cdot \Lambda\dy\right|^{p}\cdot \Lambda\dx\leq&\int_{B(1)}\int_{B(1)}\left|g(x)-g(y)\right|^p\cdot\Lambda(x)\Lambda(y)\dx\dy\\
\leq&
 \frac{2^{p+2}}{p+2} \int_{B(1)}\int_{B(1)}|\nabla g(z)|^p\cdot |z-y|^{-2}\cdot \Lambda(z)\, \mathrm{d} z\dy\\
 =& \frac{2^{p+2}}{p+2} \int_{B(1)}|\nabla g(z)|^p\cdot\Lambda(z)\int_{B(1)} |z-y|^{-2}  \dy\, \mathrm{d} z\\
 \leq&2^{p+6} \int_{B(1)}|\nabla g(z)|^p\cdot\Lambda(z)\,\mathrm{d} z.\qedhere
\end{align*}
\end{proof}

\section*{Acknowledgments}
The first author thanks the Department of Mathematics and Pacific Institute for the Mathematical Sciences in University of British Columbia, where part of this work was done.

Hui Chen was supported in part by National Natural Science Foundation of China (12101556), Natural Science Foundation of Zhejiang Province (LQ19A010002) and China Scholarship Council. Tai-Peng Tsai was supported in part by NSERC grant RGPIN-2018-04137.
	 Ting Zhang was supported in part by National Natural Science Foundation of China (11771389, 11931010, 11621101). 
\bibliography{2021ANS}
\bibliographystyle{abbrv}
\end{document}